\documentclass{amsart}

\usepackage{amsmath}
\usepackage{amssymb}
\usepackage{amsthm}

\makeatletter
\def\smallunderbrace#1{\mathop{\vtop{\m@th\ialign{##\crcr
   $\hfil\displaystyle{#1}\hfil$\crcr
   \noalign{\kern3\p@\nointerlineskip}%
   \tiny\upbracefill\crcr\noalign{\kern3\p@}}}}\limits}
\makeatother
\newcommand{\clap}[1]{\makebox[0pt]{#1}}

\usepackage[color=green!40]{todonotes}

\usepackage{geometry}               
\usepackage{mathtools}
\usepackage{tikz}
\tikzstyle{ball} = [circle,shading=ball, ball color=black,
    minimum size=1mm,inner sep=1.3pt]
\usepackage{tikz-cd}
\usepackage{enumitem}
\usepackage{url}

\newtheorem{theorem}{Theorem}

\newtheorem{cor}[theorem]{Corollary}

\theoremstyle{definition}

\newtheorem{example}[theorem]{Example}
\newtheorem{remark}[theorem]{Remark}

\newcommand{\ones}{\mathbf{1}}
\newcommand{\zeros}{\mathbf{0}}

\newcommand{\Z}{\mathbb{Z}}

\newcommand{\tL}{\widetilde{L}}

\DeclareMathOperator{\im}{\mathrm{im}}

\DeclareMathOperator{\diag}{\mathrm{diag}}

\DeclareMathOperator{\crit}{\mathcal{K}}

\DeclareMathOperator{\cok}{\mathrm{cok}}

\title{critical groups of iterated cones}
\author{Gopal Goel}
\address{Portland, OR}
\email{gopal.krishna.goel@gmail.com}
\author{David Perkinson}
\address{Reed College, Portland, OR}
\email{davidp@reed.edu}



\begin{document}

\begin{abstract} 
  Let~$G$ be a finite graph, and let~$G_n$ be the~$n$-th iterated cone over~$G$.
  We study the structure of the critical group of~$G_n$ arising in divisor
  and sandpile theory.
\end{abstract}

\maketitle

\section{Introduction} The critical group~$\crit(G)$ of a connected graph~$G$ is
the torsion part of the cokernel of its discrete Laplacian (details appear
  below).  It is known as the degree-zero part of
the Picard group or as the Jacobian of~$G$ in the divisor theory of
graphs~(\cite{Baker}).  It is isomorphic to the sandpile group
of~$G$ from statistical physics~(\cite{Dhar}) and to the group of parking
functions of~$G$ from combinatorics~(\cite{Postnikov}).  The {\em $n$-th
iterated cone} over~$G$, denoted~$G_n$, is the join of~$G$ and the complete graph
on~$n$-vertices,~$K_n$, formed by connecting each vertex of~$G$ with each vertex
of~$K_n$ by an undirected edge.  Our main result is Theorem~1, which provides a
description of the structure of~$\crit(G_n)$ as an abelian group.

The question of the structure of~$\crit(G_n)$ was addressed previously
in~\cite{Brown}.  In that paper, Theorem~A provides a short exact sequence
for~$\crit(G_n)$ (cf.~our Corollary~\ref{Theorem A}) and Corollary~B computes the
order of~$\crit(G)$ in terms of the characteristic polynomial of the Laplacian
of~$G$ (cf.~our Theorem~\ref{main}~(\ref{part3})). We give new short and direct
proofs of both of these results. We also give a partial answer to Question~1.2
of~\cite{Brown}, which asks when the short exact sequence splits.  (See the
discussion after Corollary~\ref{Theorem A}, below.) 
\medskip

\noindent{\bf Acknowledgements.} We are grateful to David Zureick-Brown for
presenting the problem of determining the structure of~$\crit(G_n)$ to us.
We thank Collin Perkinson for comments on the exposition. We also thank our
anonymous referee for helpful suggestions.

\section{Main Results}
Let~$G$ be an Eulerian digraph. As a special case,~$G$ could be an undirected
graph.  Loops and multiple edges are allowed.  We assume that~$G$ is connected
with finite vertex set~$V$ and finite edge multiset~$E$. We write~$(v,w)$ for a
directed edge starting at~$v$ and ending at~$w$. The main Eulerian property we
need is that the indegree and outdegree are equal at each vertex.
Letting~$\Z V$ denote the free abelian group on the vertices, the
(discrete) {\em Laplacian} of~$G$ is the homomorphism~$L\colon\Z V\to\Z V$ determined
by~$L(v)=\mathrm{outdeg}(v)\,v-\sum_{(v,w)\in E}w$ for each $v\in V$.  We
assume the vertices are ordered so that we can identify~$L$ with a~$k\times k$
matrix where~$k:=|V|$.  Then~$L=D-A^t$ where~$D$ is the diagonal matrix of the
outdegrees of the vertices and~$A^t$ is the transpose of the directed adjacency
matrix of~$G$.  The~$i,j$-th entry of~$A$ is the number of edges from the~$i$-th
vertex to the~$j$-th vertex.  The image of~$L$ lies in
the kernel of the ``degree'' homomorphism~$\delta\colon\Z V\to \Z V$ determined  
by~$\delta(v)=1$ for each~$v\in V$.  The {\em critical group} of~$G$ is
\[
  \crit(G):=\ker\delta/\im L.
\]
Fixing any vertex~$u\in V$, there is an isomorphism
\begin{align*}
  \cok(L)&\to\crit(G)\oplus\Z\\
  f&\mapsto (f-\delta(f)\chi_{u},\delta(f)),
\end{align*}
where~$\cok(L)$ is the cokernel of~$L$ and~$\chi_{u}\in\Z^V$ is the indicator function for~$u$.
It is well-known (e.g., via the matrix-tree theorem~(\cite[Thm.~5.6.8]{Stanley})) that since~$G$ is
connected, the rank of~$L$ is~$k-1$, and hence~$\crit(G)$ is finite.
Deleting the row and column corresponding to~$u$ from the matrix~$L$ gives the
{\em reduced Laplacian}~$\tL$ of~$G$, and since~$G$ is Eulerian
(\cite[Theorem~12.1]{Corry}), there is an
isomorphism
\[
  \crit(G)\simeq \cok(\tL)
\]
over~$\Z$.

\begin{theorem}\label{main} Let~$G$ be an Eulerian digraph with~$k$ vertices and Laplacian ~$L$.
  Let~$G_n$ be the~$n$-th cone over~$G$ where~$n\geq 2$.
  \begin{enumerate}[font=\upshape,itemsep=1em]
    \item\label{part1} Let $\ones$ be the~$k\times k$ matrix whose entries are
      all~$1$, and let~$I_k$ be the~$k\times k$ identity matrix.  Then
      \[
	\crit(G_n)\simeq \left( \Z/(n+k)\Z
	\right)^{n-2}\oplus\cok(nI_k+L+\ones).
      \]
    \item\label{part2} The group $\cok(nI_k+L+\ones)$ has a subgroup isomorphic to~$\Z/(n+k)\Z$.
    \item\label{part3} {\rm (\cite[Corollary B]{Brown})} The order of the critical group of~$G_n$ is
      \[
	|\crit(G_n)| =\frac{|p_L(-n)|}{n} (n+k)^{n-1}
      \]
      where~$p_L$ is the characteristic polynomial of~$L$.
  \end{enumerate}
\end{theorem}
\begin{proof} Order the vertices of~$G_n$ so that the cone vertices appear at
  the end.  The reduced Laplacian for~$G_n$ is then, in block form,
\[
  \tL_n = 
\left[
\renewcommand*{\arraystretch}{1.4}
\begin{array}{c|c}
    nI_k+L & -\ones \\
    \hline
    -\ones & (n+k)I_{n-1}-\ones
\end{array}
\right],
\]
where each~$\ones$ denotes a matrix of~$1$s (with dimensions inferred from
context).  Since~$G$ is Eulerian, all row and column sums of~$L$ are~$0$.
Perform the following operations in order on $\tL_n$:
\begin{enumerate}
  \item Subtract the last column from all other columns.
    \item Add all but the last row to the last row.
    \item Add the last row to all rows but the last.
\end{enumerate}
The result is the block matrix
\begin{equation}\label{diagonal}
  M:=\left[
\renewcommand*{\arraystretch}{1.4}
\begin{array}{c|c|c}
    nI_k + L + \ones & \zeros & \zeros \\
    \hline
    \zeros & (n+k)I_{n-2} & \zeros \\
    \hline
    \zeros & \zeros & 1
\end{array}
\right].
\end{equation}
Then~$\cok(M)\simeq\cok(\tL_n)\simeq\crit(G_n)$, and Part~\ref{part1} follows.

For Part~\ref{part2}, first note that since~$G_n$ is connected,~$\tL_n$ has full
rank, and hence so does~$A:=nI_k+L+\ones$. Consider the
homomorphism~$\phi\colon\Z\to\cok(A)$ sending~$1$ to the all-ones
vector~$\vec{1}$. If~$\ell\in\ker\phi$, then there is a vector~$\vec{v}\in\Z^k$
such that~$A\vec{v}=\ell\cdot\vec{1}$.  However,~$A\vec{1}=(n+k)\cdot\vec{1}$.
Since~$A$ has full rank and~$\vec{v}$ is an integer vector, it follows
that~$n+k$ divides~$\ell$ and~$\vec{v}$ is a constant vector.  Hence,~$\ker\phi$ is generated by~$n+k$.

Finally, for Part~\ref{part3}, note that~$|\crit(G_n)|=\det(M)=(n+k)^{n-2}|\det(nI_k+L+\ones)|$.
Let~$r_1,\dots,r_k$ be the rows of~$nI_k+L$.
For each~$i=1,\dots,k$, we 
use the identity~$r_1+\dots+r_k=n\vec{1}$ to substitute for~$r_i$ and use the
fact that the determinant is an alternating multilinear
function of the rows of a matrix to get 
\[
  p_L(-n)=\det(nI_k+L)=\det(r_1,\dots,r_k)=n\det(r_1,\dots,\smallunderbrace{\vec{1}}_{\clap{$i$}},\dots,r_k),
\]
where~$\vec{1}$ appears in the~$i$-th component.  Then
\begin{eqnarray*}
  \det(nI_k+L+\ones)&=&\det(r_1+\vec{1},\dots,r_k+\vec{1})\\
  &=&\det(r_1,\dots,r_k)+\sum_{i=1}^k\det(r_1,\dots,\smallunderbrace{\vec{1}}_{\clap{$i$}},\dots,r_k)\\
  &=&(n+k)\frac{p_L(-n)}{n}.
\end{eqnarray*}
The result follows.
\end{proof}
\begin{remark} To see that Corollary~B of~\cite{Brown} is equivalent to the
  result stated in Part~\ref{part3}, note that in~\cite{Brown}, the orders
  of~$H_n$ and~$\crit(G_n)$ are stated in terms of the characteristic polynomial
  of the endomorphism of~$\ker\delta$ obtained from our Laplacian by
  restriction.  Calling that characteristic polynomial~$P_G$, we
  have~$P_G(x)=p_L(x)/x$.
\end{remark}
\begin{remark} Part~\ref{part3} of the theorem also holds in the case~$n=1$,
  i.e., for the (first) cone over~$G$. The reduced Laplacian of~$G_1$
  is~$I_k+L$.  Therefore,~$\crit(G_1)\simeq\cok(I_k+L)$,
  and $|\crit(G_1)|=|\det(I_k+L)|=|p_L(-1)|$.
\end{remark}
As an immediate corollary of Theorem~\ref{main}, we have the following:
\begin{cor}\label{Theorem A} {\rm (\cite[Theorem A]{Brown})}  There is an exact sequence,
\[
  0\to\left(\Z/(n+k)\Z\right)^{n-1}\to \crit(G_n)\to H_n\to0       
\]
where~$H_n$ is a group of order~$|p_L(-n)|/n$.
\end{cor}

Question~1.2 of~\cite{Brown} asks when the exact sequence in Corollary~\ref{Theorem A}
splits.  By Theorem~1, $\left( \Z/(n+k)\Z\right)^{n-2}$ always splits off
of~$\crit(G_n)$, and the exact sequence of Theorem~A splits exactly when~$\Z/(n+k)\Z$ is a direct
summand of~$\cok(n I_k+L+\ones)$.  The latter will depend, for instance, on comparing
the prime factorization of~$n+k$ to the primary decomposition of the abelian
group~$\cok(n I_k+L+\ones)$ (cf.~Example~\ref{example1}).  It would be interesting if much more could be said
in answer to the question for arbitrary~$G$.

\begin{example}\label{example1}
  Let~$G=$ \begin{tikzpicture}\node[ball] (1) at
  (0,0) {};\node[ball] (2) at (0.5,0) {};\node[ball] (3) at (1,0)
  {};\node[ball](4) at (1.5,0) {};\draw (1)--(2)--(3)--(4);\end{tikzpicture}
  \ be the path graph on~$4$ vertices.  For this example, we 
  compute~$\crit(G_n)$ for all~$n$ and show that the exact sequence in
  Corollary~\ref{Theorem A} splits if and only if~$n$ is odd. We have
\[
  M_n:=nI_4+L+\ones=
  \left(\begin{array}{cccc}
      n+2&0&1&1\\
      0&n+3&0&1\\
      1&0&n+3&0\\
      1&1&0&n+2
  \end{array} \right).
\]
If $\diag(d_1,d_2,d_3,d_4)$ is the Smith normal form for~$M_n$, then $\cok
M_n\simeq\prod_{i=1}^4\Z/d_i\Z$.  Each~$d_i$ may be calculated as the~$\gcd$ of
the~$i\times i$ minors of~$M_n$, and~$d_i|d_{i+1}$ for~$i=1,2,3$.  Deleting 
the second row and third column from~$M_n$ produces a~$3\times 3$ submatrix matrix
with determinant~$1$.  Hence~$d_3=1$, which forces~$d_1=d_2=1$. So~$M_n$ is a
cyclic group of order $\det(M_n)$. By Theorem~\ref{main}~(\ref{part1}),
\[
  \crit(G_n)\simeq \left(\Z/(n+4)\Z\right)^{n-2}\oplus\Z/\det(M_n)\Z.
\]
Now
\[
  \det(M_n) = (n^2+4n+2)(n+4)(n+2),
\]
and thus~$\Z/\det(M_n)\Z$ contains~$\Z/(n+4)\Z$ as a direct summand if and only
if~$n+4$ is relatively prime to~$(n^2+4n+2)(n+2)$. An easy calculation shows
that
$\gcd((n^2+4n+2)(n+2),n+4)=\gcd(n,4)$, and hence~$\left(\Z/(n+4)\Z\right)^{n-1}$
is a direct summand of~$\crit(G_n)$ if and only if~$n$ is odd.

\end{example}
\bibliographystyle{amsplain}
\bibliography{cone}
\end{document}